\documentclass{amsart}

\usepackage[hidelinks]{hyperref}
\usepackage{enumerate}
\usepackage{mathrsfs}
\usepackage{graphicx,color}
\usepackage{stmaryrd}
\usepackage{amssymb}

\newtheorem{defi}{Definition}[section]
\newtheorem{prop}[defi]{Proposition}
\newtheorem{thm}[defi]{Theorem}
\newtheorem{lem}[defi]{Lemma}
\newtheorem{cor}[defi]{Corollary}
\newtheorem{ex}[defi]{Example}
\newtheorem{rem}[defi]{Remark}

\numberwithin{equation}{section}

\newcommand{\N}{\mathbb{N}}

\newcommand{\R}{\mathbb{R}}

\newcommand{\Lip}{\operatorname{Lip}}

\newcommand{\dist}{\operatorname{dist}}

\newcommand{\defl}{\mathrel{\mathop:}=}

\newcommand{\B}{\textbf B}		

\begin{document}

\title{Space of signatures as inverse limits of Carnot groups}
\author{Enrico Le Donne}
\author{Roger Z\"{u}st}

\address{\textsc{Enrico Le Donne}: \\
Dipartimento di Matematica, Universit\`a di Pisa, Largo B. Pontecorvo 5, 56127 Pisa, Italy \\
\& \\
University of Jyv\"askyl\"a, Department of Mathematics and Statistics, P.O. Box (MaD), FI-40014, Finland\\
enrico.ledonne@unipi.it}
\address{\textsc{Roger Z\"{u}st}: \\
Mathematical Institute, University of Bern, Alpeneggstrasse 22, 3012 Bern, Switzerland}


\renewcommand{\subjclassname}{%
 \textup{2010} Mathematics Subject Classification}

\thanks{E.L.D. was partially supported by the Academy of Finland (grant
288501
`\emph{Geometry of subRiemannian groups}' and by grant
322898
`\emph{Sub-Riemannian Geometry via Metric-geometry and Lie-group Theory}')
and by the European Research Council
 (ERC Starting Grant 713998 GeoMeG `\emph{Geometry of Metric Groups}').
}
\begin{abstract}
We formalize the notion of limit of an inverse system of metric spaces with $1$-Lipschitz projections having unbounded fibers. The purpose is to use sub-Riemannian groups for metrizing the space of signatures of rectifiable paths in Euclidean spaces, as introduced by Chen. The constructive limit space has the universal property in the category of pointed metric spaces with 1-Lipschitz maps. In the general setting some metric properties are discussed such as the existence of geodesics and lifts. The notion of submetry will play a crucial role. The construction is applied to the sequence of free Carnot groups of fixed rank $n$ and increasing step. In this case, such limit space is in correspondence with the space of signatures of rectifiable paths in $\R^n$. Hambly-Lyons's result on the uniqueness of signature implies that this space is a geodesic metric tree that brunches at every point with infinite valence. As a particular consequence we deduce that every path in $\R^n$ can be approximated by projections of some geodesics in some Carnot group of rank $n$, giving an evidence that the complexity of sub-Riemannian geodesics increases with the step.

\end{abstract}

\maketitle


\section{Introduction}
In the study of rough paths in stochastic PDE theory to each path one assigns its signature via the method of iterated integrals. This concept was introduced by Chen \cite{C, C2} who aimed to give an algebraic structure to the space of paths. The space of $k$-truncated signatures of rectifiable paths 
in $\R^n$  has the structure of the free nilpotent Lie group of rank $n$ and step $k$. In fact, the truncated signatures have a sub-Riemannian meaning. Namely, the above Lie group has a natural structure of a Carnot group and hence to every rectifiable path in $\R^n$ one assigns a horizontal lift starting from the identity in the Carnot group. The truncated signature of the path is the final point of the lifted path.

With the idea that the signature of a rectifiable path lies in an infinite dimensional tensor algebra, one would like to metrize the space of signatures as a limit of Carnot groups. The first aim of this paper is to formalize the appropriate notion of limit of inverse systems of metric spaces. Our construction differs from previous ones as for example in \cite{CK} by the fact that the fibers of the bonding maps may not have finite diameter and it is inspired by the construction of ultralimits of metric spaces. Namely, let $((X_i, \star_i)_{i\in I}, (\pi^i_{j})_{i\geq j\in I})$ be an inverse system of pointed metric spaces and $1$-Lipschitz base-point-preserving maps $\pi^i_{j}: (X_i, \star_i) \to (X_j, \star_j)$ as connecting morphisms. In Lemma~\ref{inverse_lem} we show that the system admits an inverse limit $((X_\infty, \star_\infty), (\pi^\infty_{i})_{i\in I})$ in this category and points in $X_\infty$ can be characterized by those elements $(x_i)_{i\in I} \in \prod_{i \in I} X_i$ such that 
\[
\pi^i_{j}(x_i) = x_j \text{ for all } i\geq j \quad \text{and} \quad \sup_{i\in I} d(\star_i, x_i) < \infty \ .
\]
Our interest will be in the inverse limit of the system $(G^k(\R^n), \pi_k)_{k\in \N}$, where $G^k(\R^n)$ is the free Carnot group of rank $n$ and step $k$, pointed at the identity element. The map $\pi_k: G^k(\R^n) \to G^{k-1}(\R^n)$ is the projection modulo the $k$-layer, i.e., the $k$-th element of the lower central series of $G^k(\R^n)$. The maps $\pi_k$ have two properties that will be crucial in our following discussion: they are submetries (see Definition \ref{submetry_def}) and have lifting properties for rectifiable paths (see Definition~\ref{def:uniquelift} and \ref{def:continuouslift}). Under this type of assumptions we can draw several conclusions on our inverse limit space $X_\infty$. For example, also the projections maps $\pi^\infty_{i}$ are submetries with lifting properties and the metric space $X_\infty$ is a geodesic space, if so are the $X_i$'s (see Lemmas~\ref{uniqueliftinfinity} and Proposition~\ref{geodesic_prop}, respectively).

We deduce that the space of signatures of rectifiable paths in $\R^n$ has a canonical identification with the metric space 
\begin{equation}
\label{limit_Fnk}
\varprojlim _{k\in \N} \bigl(G^k(\R^n), \pi_k\bigr) \ .
\end{equation}
We draw some consequences from a deep result of Hambly and Lyons \cite{HL} on the uniqueness of signatures, i.e., every rectifiable path in $\R^n$ is completely determined by its signature, up to reparametrization, translation, and tree like reduction. In Theorem~\ref{tree_thm} we show that the space \eqref{limit_Fnk} is a geodesic metric tree. In fact, it is a metric space homeomorphic to a  tree that brunches at every point with infinite valence. Consequently we deduce a particular consequence about the complexity of sub-Riemannian geodesics in Carnot groups. Indeed, we come to the conclusion that every path in $\R^n$ can be arbitrarily approximated by the projection of some geodesic in $G^k(\R^n)$ for some $k$ (see Corollary~\ref{approx_cor}).
 
One more purpose to study the space \eqref{limit_Fnk} was with the aim of properly define and metrize the free Lie group with $n$ generators. Clearly, the space \eqref{limit_Fnk} still has an induced group structure for which the distance is left invariant (we initially take Carnot groups with left invariant distances). Unfortunately, this space is not a topological group since the right translations are not continuous anymore. We prove this last claim in general. Namely, a group with the topology of a tree cannot have continuous left and right translations simultaneously, unless it is $\R$ (see Theorem~\ref{treenogroup_thm}).

\section{Submetries and path lifting properties}

\label{section:inverselimit}
 
Let $(I, \geq)$ be a partially ordered set. For each $i\in I$, let $X_i$ be a metric space and let $\star_i \in X_i$ a choice of a base point in $X_i$. We shall say that $(X_i, \star_i)$ is a pointed metric space. For all $i, j\in I$ with $i\geq j$ let $\pi^i_{j}: (X_i, \star_i) \to (X_j , \star_j)$ be a map that is $1$-Lipschitz and preserves the base point, i.e.,  $\pi_{i+1}(\star_{i+1}) = \star_i$. The maps  $\pi^i_{j}$ are called {\bf bonding maps} if $\pi^i_{i}$ is the identity map in $X_i$ and $\pi^j_{k}\circ \pi^i_{j}=\pi^i_{k}$, for all $i\geq j\geq k$. In this case, we have that $((X_i, \star_i)_{i\in I}, (\pi^i_{j})_{i\geq j\in I})$ is an {\bf inverse system} of pointed metric spaces and $1$-Lipschitz base-point-preserving maps $\pi^i_{j}: (X_i, \star_i) \to (X_j , \star_j)$.

The most important example for us is when $I=\N$. Namely, assume we have a sequence of pointed metric spaces $(X_i, d_i, \star_i)_{i \in \N}$ together with $1$-Lipschitz maps $\pi_{i+1} : X_{i+1} \to X_{i}$ that preserve the base point. This leads to  an inverse system. For $i \geq j$, we define 
\begin{equation}
\label{def:comp}
\pi_{j}^i : X_i \to X_j \quad \text{as} \quad \pi_{j}^i \defl \pi_{j+1} \circ \cdots \circ \pi_{i} \ .
\end{equation}
As a composition of $1$-Lipschitz maps, each $\pi_{j}^i$ is itself $1$-Lipschitz.

As discussed in the introduction, we consider the space $X_\infty$ as the set of elements $(x_i)_{i\in I} \in \prod_{i \in I} X_i$ with $x_i \in X_i$ with the property that $\pi_j^{i}(x_{i}) = x_j$ and
\[
\sup_{i \in I} d_i(\star_i,x_i) < \infty \ .
\]
It is thus clear that $X_\infty$ is nonempty since $\star_\infty \defl (\star_i)_{i\in I} \in X_\infty$. The distance between $x = (x_i)_{i\in I}$ and $y = (y_i)_{i\in I}$ in $X_\infty$ is defined by  
\[
d_\infty(x,y) \defl \sup_{i \in I}d_i(x_i,y_i) \ .
\]
There are natural projections $\pi_i^\infty: X_\infty \to X_i$ is given by $\pi_i^\infty((x_j)_{j\in I}) \defl x_i$ for $i \in I$.

\begin{defi}
	\label{submetry_def}
A \textbf{submetry} is a $1$-Lipschitz map between (nonempty) metric spaces $\pi : X \to Y$ such for any $y,y' \in Y$ and $x \in X$ with $\pi(x) = y$ there is a $x' \in X$ with $\pi(x') = y'$ and $d_X(x,x') = d_Y(y,y')$.
\end{defi}

If in the notation \eqref{def:comp} each of the $\pi_i$ is a submetry, then also every $\pi_{i}^j$ is a submetry. This last claim follows from the simple observation that the composition of submetries is also a submetry. 

\begin{lem}
\label{firstlem} The space $(X_\infty,d_\infty)$ constructed above is a metric space and further:
\begin{enumerate}
	\item The projection $ \pi_i^\infty : (X_\infty,d_\infty) \to (X_i,d_i)$ is $1$-Lipschitz for every $i \in I$.
	\item If $X_i$ is complete for every $i \in I$, then so is $X_\infty$.
	\item If $I=\N$, then $d_\infty(x , y) = \lim_{i \to \infty} d_i(   \pi_i^\infty(x),   \pi_i^\infty(y))$.
	\item If $I=\N$ and $\pi_i$ is a submetry for every $i$, then $ \pi_i^\infty$ is a submetry for every $i \in \N$. 
\end{enumerate}
\end{lem}

\begin{proof}
Let $x = (x_i)_{i \in I}$, $y = (y_i)_{i \in I}$ and $z = (z_i)_{i \in I}$ be arbitrary elements of $X_\infty$. First note that by the triangle inequality in each $X_i$, we have
\[
d_\infty(x,y) \leq \sup_{i \in I} d_i(\star_i,x_i) + d_i(\star_i,y_i) \leq \sup_{i \in I}d_i(\star_i,x_i) + \sup_{i \in I} d_i(\star_i,y_i) < \infty \ .
\]
Hence $d_\infty$ is finite. Similarly one can show the triangle inequality. Indeed,
\begin{align*}
d_\infty(x, z) & =  \sup_{i \in I} d_i(x_i,z_i) \leq \sup_{i \in I} d_i(x_i,y_i) + \sup_{i \in I} d_i(x_i,z_i) \\
 & \leq d_\infty(x, y) + d_\infty(y, z) \ .
\end{align*}
If $d_\infty(x, y) = 0$, then of course $x_i = x_i'$ for all $i \in I$ and hence $x = y$. Thus $d_\infty$ is a metric on $X_\infty$.

(1) is clear since $d_i(x_i,y_i) \leq d_\infty(x,y)$ for all $i \in \N$.

For (2) consider a Cauchy sequence $(x^n)_{n \in \N}$ in $X_\infty$. We can write $x^n = (x^n_i)_{i \in I}$ for every $n \in \N$. Because $\pi^\infty_i : X_\infty \to X_i$ is $1$-Lipschitz by (1), the sequence $(x^n_i)_{n \in \N}$ is a Cauchy sequence in $X_i$ for every $i \in I$. By assumption these sequences converge, say to $x_i \defl \lim_{n \to \infty} x^n_i \in X_i$. We want to show that $(x^n)_{n \in \N}$ converges to $x \defl (x_i)_{i \in I}$ in $X_\infty$. First note that $\pi^i_j(x_i) = x_j$ for $i \geq j$ because $\pi^i_j$ is continuous and $\pi^i_j(x_i^n) = x_j^n$ holds for all $n \in \N$. Because $(x^n)_{n \in \N}$ is a Cauchy sequence there is some $R > 0$ such that $x^n \in \B(\star_\infty,R)$ for all $n \in \N$, which is equivalent to $d_i(\star_i,x_i^n) \leq R$ for all $i$ and $n$. Taking limits, this implies that also $d_i(\star_i,x_i) \leq R$ for all $i$ and hence $x$ is an element of $X_\infty$. Now if $\epsilon > 0$ is given, there exists $N_\epsilon \in \N$ such that $d_\infty(x^n,x^m) \leq \epsilon$ for all $n,m \geq N_\epsilon$. Again because each $\pi^\infty_i$ is $1$-Lipschitz it holds that $d_i(x_i^n,x_i) \leq \epsilon$ for all $i \in I$ and all $n \geq N_\epsilon$. Hence 
\[
d_\infty(x^n,x) = \sup_{i \in I} d_i(x^n_i,x_i) \leq \epsilon
\]
for all $n \geq N_\epsilon$. Thus $x^n$ converges to $x$ in $X_\infty$ and $(2)$ follows.

To see (3) note that $\pi_{i}(x_{i+1}) = x_{i}$ and $\pi_{i}(x_{i+1}') = x_{i}'$. Because each $\pi_{i}$ is $1$-Lipschitz, the sequence $(d_i(x_i,y_i))_{i \in \N}$ is increasing. Hence (3) holds.

In order to show (4) fix $k \in I$ and consider $x_k,y_k \in X_k$ and $x = (x_i)_{i \in \N} \in X_\infty$, i.e.\ such that $\pi^\infty_k(x) = x_k$. For $i < k$ set $y_i \defl \pi^k_{i}(y_k)$. For $i > k$ we recursively define $y_{i+1} \in \pi_{i+1}^{-1}(y_i)$ to be some point with $d_{i+1}(x_{i+1},y_{i+1}) = d_{i}(x_i,y_i)$. This is possible because each $\pi_{i}$ is a submetry. By construction and (3),
\begin{equation}
\label{sup_eq}
\sup_{i\in\N} d_{i}(x_{i},y_{i}) = \lim_{i \to \infty} d_i(x_i,y_i) = d_k(x_k,y_k) < \infty \ .
\end{equation}
Set $y \defl (y_i)_{i \in \N} \in \prod_{i \in \N} X_i$. Since 
\[
\sup_{i\in\N} d_{i}(\star_i, y_{i}) \leq \sup_{i\in\N} d_{i}(\star_i, y_{i}) + \sup_{i\in\N} d_{i}(x_{i},y_i) < \infty
\]
it follows that $y \in X_\infty$. From \eqref{sup_eq} we obtain 
\[
d_\infty(x,y) = \lim_{i \to \infty} d_i(x_i, y_i) = d_k(x_k,y_k) \ .
\]
Since each $\pi^\infty_k$ is $1$-Lipschitz by (1), this shows that these maps are submetries.
\end{proof}

We readily show that the space $(X_\infty,\star_\infty)$, as defined, is the inverse limit in the category of pointed metric spaces with base-point-preserving $1$-Lipschitz maps. Notice that we do not show the analogue considering submetries, see Example~\ref{ex:counterexample} below.

\begin{lem}
\label{inverse_lem}
Assume that $(X_i,\star_i)_{i\in I}$ together with $\pi^i_{j}:X_i \to X_j$ for $i \geq j$ is an inverse system in the category of pointed metric spaces with base-point-preserving $1$-Lipschitz maps as morphisms. Then $(X_\infty,\star_\infty)$ is its inverse limit.
\end{lem}

\begin{proof}
We need to check the universal property for $(X_\infty,\star_\infty)$. In this direction let $(Y,\ast)$ be some pointed metric space together with base-point-preserving $1$-Lipschitz maps $u_i : Y \to X_i$ that satisfy $\pi^i_{j}\circ u_i = u_j$ for all $i \geq j$. We need to find a base-point-preserving $1$-Lipschitz $u : Y \to X_\infty$ such that $\pi^\infty_i \circ u = u_i$ for all $i \in I$. For $y \in Y$ set $u(y) \defl (u_i(y))_{i \in I} \in \prod_{i \in I} X_i$. Since each $u_i$ is base-point-preserving, $u(\ast) = (\star_i)_{i \in I} = \star_\infty$. Since each $u_i$ is $1$-Lipschitz,
\begin{align*}
\sup_{i \in I} d_i(\star_i,u_i(y)) = \sup_{i \in I} d_i(u_i(\ast),u_i(y)) \leq d_Y(\ast, y) < \infty \ .
\end{align*}
This shows that $u$ is a base-point-preserving map into $X_\infty$. Moreover, the function $u$ is $1$-Lipschitz for the same reason: given $y,y' \in Y$, then
\[
d_\infty(u(y),u(y')) = \sup_{i \in I} d_i(u_i(y),u_i(y')) \leq d_Y(y,y') \ .
\]
Hence $u : Y \to  X_\infty$ is a base-point-preserving $1$-Lipschitz map. It is clear that $u : Y \to X$ is the only map that satisfies $\pi^\infty_i \circ u = u_i$ for all $i \in I$.
\end{proof}

It is interesting to note that $(X_\infty,\star_\infty)$ is not the inverse limit in the category of pointed metric spaces with base-point-preserving submetries. Indeed, in this category inverse limits do not exist in general as the following example demonstrates. This does not come as a surprise because submetries being surjective is a rather strong restriction.

\begin{ex}
	\label{ex:counterexample}
For each $n \in \N$ let $X_n$ be the metric space with $n$ elements denoted by $\{a_1,\dots,a_n\}$ and pairwise distances $1$ for different points. The projection $\pi_{n+1} : X_{n+1} \to X_n$ is given by $\pi_{n+1}(a_i) = a_i$ for $i \leq n$ and $\pi_{n+1}(a_{n+1}) = a_n$. These maps are clearly submetries. Let $Y = \{a_i,b_i\}_{i \in \N} \cup \{b\}$ be a countable metric space with distances $d(a_i,a_j) = d(b_i,b_j) = d(b_i,b) = 1$ for $i \neq j$ and $d(a_i,b_j) = d(a_i,b) = 2$ for all $i,j\in \N$. Let $u_n : Y \to X_n$ be the map given by $u_n(a_i) = u_n(b_i) = a_i$ for $i \leq n$ and $u_n(y) = a_n$ for all other points $y \in Y$. These maps $u_n$ are compatible with the projections $\pi_n$ and submetries. For the latter note that for different fibers in $Y$ have distance $\dist(u_n^{-1}(a_i),u_n^{-1}(a_j)) = 1 = d_n(a_i,a_j)$ and this distance is realized for any point in one of the fibers.
	
Now assume $u : Y \to X$ is a map into a metric space such that $u_n = \pi^\infty_n \circ u$ for all $n \in \N$. We claim that $u$ can not be a submetry. First note that $u^{-1}(u(b)) = \{b\}$ because for any other point $y \neq b$ in $Y$ there exists some $n \geq 1$ such that $u_n(y) \neq u_n(b)$. For the same reason, the fiber $u^{-1}(u(a_i))$ is either $\{a_i\}$ or $\{a_i,b_i\}$. Assume first that there exists some $i$ such that $u^{-1}(u(a_i)) = \{a_i,b_i\}$. Then 
\[
\dist\bigl(u^{-1}(u(a_i)),u^{-1}(u(b))\bigr) = \dist(\{a_i,b_i\},\{b\}) = 1
\]
but this distance is not realized by the point $\{a_i\}$ in the first fiber. In the other case we do have $u^{-1}(u(a_i)) = \{a_i\}$ for all $i$ and hence also $u^{-1}(u(b_i)) = \{b_i\}$ for all $i$. Thus $u : Y \to X$ as an injective submetry is an isometry. If we do the same for $Y'$ obtained from $Y$ by changing all the distances with value $2$ to $3/2$ we obtain similarly that $Y'$ is isometric to $X$, but $Y$ and $Y'$ are obviously not isometric. This shows that there can not be an inverse limit for the sequence of spaces given and assuming the morphisms are submetries.
\end{ex}

In what follows we introduce additional properties for maps.

\begin{defi}
	\label{def:uniquelift}
A submetry $\pi : X \to Y$ has the \textbf{unique path lifting property} if for all rectifiable paths $\gamma : [a,b] \to Y$ and each point $x \in \pi^{-1}(\gamma(a))$ there is path $\bar\gamma : [a,b] \to X$ such that
\begin{enumerate}
	\item $L(\bar\gamma) < \infty$,
	\item $\bar \gamma(a) = x$,
	\item $\pi(\bar\gamma(t)) = \gamma(t)$ for all $t \in [a,b]$,
	\item if a continuous path $\eta : [a,b] \to X$ satisfies (1),(2) and (3), then $\eta = \bar \gamma$,
	\item $L(\bar\gamma) \leq L(\gamma)$ (because $\pi$ is $1$-Lipschitz we have equality).
\end{enumerate}
Such a path $\bar \gamma$ is called a \textbf{lift} of $\gamma$.
\end{defi}

\begin{lem}
\label{uniqueliftproperties}
Assume that $\pi : X \to Y$ has the unique path lifting property and $\bar \gamma$ is a lift of a rectifiable path $\gamma : [a,b] \to Y$. The following hold:
\begin{enumerate}[$\qquad (a)$]
	\item If $a \leq s < t \leq b$, then $L(\bar\gamma|_{[s,t]}) = L(\gamma|_{[s,t]})$.
	\item If $\gamma$ is $L$-Lipschitz, then also $\bar\gamma$ is  $L$-Lipschitz.
	\item If $\eta : [a,b] \to X$ is rectifiable, then $\eta$ is the unique lift of $\pi \circ \eta$ starting at $\eta(S)$.
\end{enumerate}
\end{lem}

\begin{proof}
The path $\bar\gamma|_{[s,t]}$ satisfies (1), (2) and (3) in the definition above for the interval $[s,t] \subset [a,b]$ and the starting point $\bar \gamma(s)$. Hence $\bar\gamma|_{[s,t]}$ is the unique lift with these properties by (4) and  hence $L(\bar\gamma|_{[s,t]}) = L(\gamma|_{[s,t]})$ by (5). This shows (a). (a) also follows directly from $\Lip(\pi) \leq 1$ and $L(\bar\gamma) = L(\gamma)$.

If moreover $\gamma$ is $L$-Lipschitz, then
\[
d_X(\bar\gamma(s),\bar\gamma(t)) \leq L(\bar\gamma|_{[s,t]}) \leq L(\gamma|_{[s,t]}) \leq L|t-s| \,.
\]
Hence also $\bar\gamma$ is  $L$-Lipschitz.

Since $\pi$ is Lipschitz, also the projection $\pi \circ \bar \gamma$ is  rectifiable. It is clear that $\bar \gamma$ is a lift of $\pi \circ \bar \gamma$ and the uniqueness follows from the assumption on $\pi : X \to Y$.
\end{proof}

Next we show that the unique path lifting property is inherited by $X_\infty$.

\begin{lem}
\label{uniqueliftinfinity}
Assume that $\pi_{i+1} : (X_{i+1},\star_{i+1}) \to (X_{i},\star_{i})$ is base-point-preserving and has the unique path lifting property for every $i \in \N$. Then $\pi^\infty_{i} : X_{\infty} \to X_i$ also has the unique path lifting property for every $i \in \N$.
\end{lem}

\begin{proof}
From Lemma~\ref{firstlem} we already know that each $\pi^\infty_i$ is a submetry. Fix some $j \in \N$ and let $\gamma_j : [a,b] \to X_j$ be a rectifiable path and let $x = (x_1,x_2,\dots) \in X_\infty$ be a point with $x_j = \gamma_j(a)$. Such a point exists because $\pi^\infty_j$ as a submetry is surjective. For $i < j$, let $\gamma_i \defl \pi^j_{i}\circ\gamma_j$. For $i > j$ let $\gamma_i$ be the successive lift of $\gamma_j$ with $\gamma_i(a) = x_i$. First we show that $\gamma_\infty(t) \defl (\gamma_1(t),\gamma_2(t),\dots)$ is in $X_\infty$ for each $t \in [a,b]$. Because lifts are length-preserving,
\[
d_i(x_i,\gamma_i(t)) \leq L(\gamma_i) \leq L(\gamma_j) \ ,
\]
for all $i > j$ and the same holds true for $i < j$ because each $\pi_i$ is $1$-Lipschitz. Hence
\[
\sup_{i \geq 1} d_i(\star_i,\gamma_i(t)) \leq \sup_{i \geq 1} d_i(\star_i,x_i) + \sup_{i \geq 1} d_i(x_i,\gamma_i(t)) < \infty \ ,
\]
which implies that $\gamma_\infty(t) \in X_\infty$ for all $t \in [a,b]$. From Lemma~\ref{uniqueliftproperties} it follows that $L(\gamma_i|_{[s,t]}) = L(\gamma_j|_{[s,t]})$ for all $a \leq s < t\leq b$ and $i \in \N$. Further
\begin{align*}
d_\infty(\gamma_\infty(s),\gamma_\infty(t)) & = \sup_{i \geq 1} d_i(\gamma_i(s),\gamma_i(t)) \leq \sup_{i \geq 1} L(\gamma_i|_{[s,t]}) \\
 & \leq L(\gamma_j|_{[s,t]}) \,.
\end{align*}
This shows that $\gamma_\infty$ is continuous since $L(\gamma_j|_{[s,t]}) \to 0$ for $t-s \to 0$. By taking a finite partition of $[a,b]$ and summing up the distances it is immediate that $L(\gamma_\infty) \leq L(\gamma_j)$. So $\gamma_\infty$ is a lift of $\gamma_j$ and it remains to show the uniqueness of $\gamma_\infty$. 

Assume that $\eta : [a,b] \to X_\infty$ is a rectifiable path with $\bar\pi_j(\eta(t)) = \gamma_j(t)$ for all $t$ as well as $\eta(a) = x$. Then for each $i > j$, the path $\pi^\infty_i \circ \eta$ is a lift of $\gamma_j$ with $\pi^\infty_i \circ \eta(a) = x_i$. By the unique path lifting property it follows that $\pi^\infty_i \circ \eta(t) = \pi^\infty_i \circ \gamma_\infty(t)$ for all $t$ and $i$. Hence $\eta = \gamma_\infty$ and we are done.
\end{proof}

It follows from Lemma~\ref{uniqueliftproperties}(c) that $\gamma_\infty : [a,b] \to X_\infty$ is uniquely defined by $\gamma_\infty(a)$ and $\pi^\infty_i\circ \gamma_\infty$ for any given $i$. We call such a rectifiable path $\gamma_\infty$ the lift of its projections $\pi^\infty_i \circ \gamma_\infty$ at the starting point $\gamma_\infty(a)$.

\begin{defi}
	\label{def:continuouslift}
A map $\pi : X \to Y$ has the \textbf{continuous path lifting property} if it has the unique path lifting property and for each sequence of paths $\gamma_n : [0,1] \to Y$, $n \in \N$, with
\begin{enumerate}
	\item $\gamma_n(0) = y$ for all $n$,
	\item $\sup_{n \geq 1} \Lip(\gamma_n) < \infty$,
	\item $\gamma(t) = \lim_{n\to\infty} \gamma_n(t)$ for all $t \in [0,1]$,
\end{enumerate}
it holds that
\[
\bar\gamma(t) = \lim_{n\to\infty} \bar \gamma_n(t) \quad \mbox{for all}\quad t \in [0,1] \ ,
\]
where $\bar \gamma_n$ and $\bar \gamma$ are lifts of $\gamma_n$ and $\gamma$, respectively, with the same starting point $\bar \gamma_n(0) = \bar \gamma(0) \in \pi^{-1}(y)$.
\end{defi}

Note that in the definition above, assuming uniform bounds on Lipschitz constants, pointwise convergence and uniform convergence are equivalent.

Recall that a metric space $(X,d)$ is \textbf{proper} if all the closed balls $\B(x,r)$ are compact. The metric space $(X,d)$ is a \textbf{geodesic space} if for all points $x,y \in X$ there is a path $\gamma : [a,b] \to X$ with $\gamma(a)=x$, $\gamma(b)=y$ and $L(\gamma) = d(x,y)$. Any such path is called a geodesic. A curve $\gamma$ is \textbf{parametrized proportionally to arc length} if $L(\gamma|_{[s,t]}) = L(\gamma)|b-a|^{-1}|t-s|$ for all $a \leq s \leq t \leq b$.

\begin{prop}
\label{geodesic_prop}
Assume that for every $i \in \N$, the pointed metric space $(X_i,\star_i)$ is geodesic and proper. If all the projections $\pi_{i+1} : X_{i+1} \to X_{i}$ are base-point-preserving and have the continuous path lifting property, then $X_\infty$ is a complete geodesic metric space and for each pair of points $x = (x_1,x_2,\dots), y = (y_1,y_2,\dots) \in X_\infty$ there is $\eta_\infty : [0,1] \to X_\infty$ and $\gamma^{\sigma(n)} : [0,1] \to X_{\sigma(n)}$ for a strictly increasing sequence $\sigma : \N \to \N$ such that:
\begin{enumerate}
	\item [(A)] $\eta_\infty$ is a geodesic parametrized proportionally to arc length connecting $x$ with $y$ in $X_\infty$.
	\item [(B)] Each $\gamma^{\sigma(n)}$ is a geodesic parametrized proportionally to arc length connecting $x_{\sigma(n)}$ with $y_{\sigma(n)}$,
	\item [(C)] $\pi^\infty_i \circ \eta_\infty = \lim_{n \to \infty} \pi^{\sigma(n)}_i \circ \gamma^{\sigma(n)}$ for all $i \in \N$.
\end{enumerate}
\end{prop}

\begin{proof}
Note that proper metric spaces are in particular complete, so it already follows from Lemma~\ref{firstlem} that $X_\infty$ is a complete metric space. For each $n \in \N$ let $\gamma^n : [0,1] \to X_k$ be a geodesic parametrized proportional to arc length connecting $x_n$ with $y_n$. This implies that $\Lip(\gamma^n) \leq d_n(x_n,y_n)$ for all $n \in \N$. From Lemma~\ref{uniqueliftinfinity} it follows that the lift $\gamma_\infty^n : [0,1] \to X_\infty$ of $\gamma^n$ with $\gamma_\infty^n(0) = x$ satisfies
\[
\Lip(\gamma_\infty^n) \leq \Lip(\gamma^n) \leq d_n(x_n,y_n) \leq d_\infty(x,y) \ .
\]
Hence the path $\gamma_\infty^n$ is contained in the closed ball $\B(x,d_\infty(x,y))$. Since each $\pi^\infty_i$ is $1$-Lipschitz by Lemma~\ref{firstlem}, the path $\pi^\infty_i \circ \gamma_\infty^n$ is contained in $\B(x_i,d_\infty(x,y))$, starts at $x_i$ and satisfies $\Lip(\pi^\infty_i \circ \gamma_\infty^n) \leq d_\infty(x,y)$ for all $i$ and $n$. Since $\B(x_1,d_\infty(x,y))$ is compact, the theorem of Arzel\`a-Ascoli guarantees a subsequence $(\sigma(1,n))_{n\in\N}$ of $(n)_{n\in\N}$ such that $(\pi^\infty_1 \circ \gamma_\infty^{\sigma(1,n)})_{n\in\N}$ converges uniformly to some path $\eta_1$. Recursively we find for every $i \geq 2$ a subsequence $(\sigma(i,n))_{n\in\N}$ of $(\sigma(i-1,n))_{n\in\N}$ such that $(\pi^\infty_i \circ \gamma_\infty^{\sigma(i,n)})_{n\in\N}$ converges uniformly to some path $\eta_i$. For the diagonal sequence it holds that
\[
\lim_{n \to \infty} \pi^\infty_i \circ \gamma_\infty^{\sigma(n,n)} = \eta_i
\]
uniformly for all $i \in \N$. We have the following properties of these limits for all $i \in \N$:
\begin{enumerate}
	\item $\eta_i(0) = x_i$,
	\item$\eta_i(1) = y_i$,
	\item $\Lip(\eta_i) \leq d_\infty(x,y)$,
	\item $\pi_{i+1}(\eta_{i+1}) = \eta_i$.
\end{enumerate}
(1) is clear by construction since $\pi^\infty_i \circ \gamma_\infty^n(0) = x_i$ for all $i,n \in \N$. Similarly, (2) follows from the fact that $\pi^\infty_i \circ \gamma_\infty^n(1) = y_i$ for all $i \leq n$. (3) is a consequence of $\Lip(\pi^\infty_i \circ \gamma_\infty^n) \leq d_\infty(x,y)$ for all $i,n \in \N$. To see (4) note that the path $\pi^\infty_{i+1} \circ \gamma_\infty^{\sigma(n,n)}$ is the unique lift of $\pi^\infty_{i} \circ \gamma_\infty^{\sigma(n,n)}$ starting at $x_{i+1}$ for all $i,n \in \N$. The continuous path lifting property now implies that $\eta_{i+1}$ is the unique lift of $\eta_i$ for each $i \in \N$. Hence (4) holds.

It now follows from (1),(3),(4) and Lemma~\ref{uniqueliftinfinity} that $\eta_\infty(t) = (\eta_1(t),\eta_2(t),\dots)$ is for each $i \in \N$ the unique lift of $\eta_i$ starting at $x$. Moreover $\Lip(\eta_\infty) \leq d_\infty(x,y)$ by (3) and Lemma~\ref{uniqueliftproperties}. The equality $\eta_\infty(1) = y$ holds because of (2) and hence $\eta_\infty$ is a geodesic connecting $x$ with $y$. The bound $\Lip(\eta_\infty) \leq d_\infty(x,y)$ also implies that $\eta_\infty$ is parametrized proportional to arc length.
\end{proof}

Note that in the setting of the proposition above the projections $\pi^{\infty}_i : X_\infty \to X_i$ in general do not have the continuous path lifting property as we will demonstrate in Remark~\ref{rem:limitcontpathlifing} below.

\section{Group structures and their limits}

 We consider the case where each metric space $X_i$ has a group structure making the distance left invariant. The following lemma shows that then also the metric space $X_\infty$ admits a group structure making the distance left invariant.

\begin{lem}
\label{grouplem}
Assume that $(X_i,e_i)_{i \in I}$ is a collection of pointed metric groups together with maps $\pi^{i}_j : X_{i} \to X_{j}$ for $i \geq j$ such that:
\begin{enumerate}
	\item [(A)] $e_i$ is the identity element of $X_i$,
	\item [(B)] the projections $\pi^{i}_j : X_{i} \to X_{j}$ are homomorphisms and submetries,
	\item [(C)] left translations are isometries, i.e.,  for each $g_i,x_i,y_i \in X_i$ it holds that
\[
d_i(g_i \cdot x_i, g_i \cdot y_i) = d_i(x_i, y_i) \ .
\]
\end{enumerate}
Then the limit $(X_\infty,e_\infty)$ is a group with identity element $e_\infty = (e_i)_{i \in I}$, and operations $(x_i)_{i \in I}^{-1} \defl (x_i^{-1})_{i \in I}$ as well as $(x_i)_{i \in I} \cdot (y_i)_{i \in I} \defl (x_i \cdot y_i)_{i \in I}$. Moreover, left translations in $X_\infty$ are isometries with respect to $d_\infty$.
\end{lem}

\begin{proof}
We need to show that
\begin{enumerate}
	\item $e_\infty \in X_\infty$,
	\item $x^{-1} \in X_\infty$ if $x \in X_\infty$,
	\item $x \cdot y \in X_\infty$ if $x, y \in X_\infty$,
	\item $d_\infty(g \cdot x, g \cdot y) = d_\infty(x, y)$ for all $g, x, y \in X_\infty$.
\end{enumerate}
(1) is clear by definition. To see (2) let $x = (x_i)_{i \in I} \in X_\infty$. By assumption it holds that
\[
\sup_{i\in I} d_i(e_i, x_i^{-1}) = \sup_{i\in I} d_i(x_i \cdot e_i, x_i \cdot x_i^{-1} \cdot e_i) = \sup_{i\in I} d_i(x_i, e_i) < \infty \ .
\]
Because each $\pi^i_j$ is a homomorphism $\pi^{i}_j(x_i^{-1}) = \pi^{i}_j(x_i)^{-1} = x_i^{-1}$ and this implies that $(x_i)_{i \in I}^{-1} \in X_\infty$. For (3) note that by left invariance of each $d_i$ and by (2) it holds
\[
\sup_{i\in I} d_i(e_i, x_i \cdot y_i) = \sup_{i\in I} d_i(x_i^{-1}, y_i) \leq \sup_{i\in I} d_i(x_i^{-1}, e_i) + d_i(e_i, y_i) < \infty \ .
\]
Moreover, $\pi^{i}_j(g_{i}\cdot x_{i}) = \pi^{i}_j(g_{i})\cdot \pi^{i}_j(x_{i}) = g_j \cdot x_i$ again because all $\pi^i_j$ are homomorphisms. Thus $(x_i \cdot y_i)_{i \in I} \in X_\infty$. Clearly, 
\begin{align*}
d_\infty(g \cdot x, g \cdot y) & = \sup_{i\in I} d_i(g_i \cdot x_i, g_i \cdot y_i) = \sup_{i\in I} d_i(x_i, y_i) \\
 & = d_\infty(x, y) \ ,
\end{align*}
and shows (4).
\end{proof}

We will see later in Theorem~\ref{treenogroup_thm} that right translations in $X_\infty$ are in general not continuous, and the resulting limit space thus is not a topological group. A particular example of this behaviour is given by the sequence $X_i = G^i(\R^n)$ of free nilpotent Lie groups of step $i$ over $\R^n$ equipped with the Carnot-Carath\'eodory metric.

\section{Space of signatures as an inverse limit}

We first review the basic definitions and results about signatures and free nilpotent Lie groups, we mainly refer to \cite{FV}. As defined in \cite[Definition~7.2]{FV} the \textbf{step} $\mathbf k$ \textbf{signature} of a rectifiable path $\gamma : [a,b] \to \R^n$ of is given by
\begin{align*}
S_k(\gamma)_{a,b} & \defl \left(1,\int_{a < t < b} d\gamma_t, \dots, \int_{a < t_1 < \cdots < t_k < b} d\gamma_{t_1} \otimes \cdots \otimes d\gamma_{t_k}\right) \\
 & \in T^k(\R^n) \defl \prod_{i=0}^k(\R^n)^{\otimes i} \ .
\end{align*}
Note that the iterated integrals are interpreted as Riemann-Stieltjes integrals and thus well defined and independent of the particular parametrization of $\gamma$. There are natural coordinate projections $\xi_k : T^l(\R^n) \to (\R^n)^{\otimes k}$ and $\pi_{k}^l : T^l(\R^n) \to T^k(\R^n)$ for $0 \leq k \leq l$. For $g,h \in T^k(\R^n)$ the truncated tensor product is defined by
\[
g \otimes_k h \defl \sum_{0 \leq i+j \leq k} \xi_i(g) \otimes \xi_j(h) \in T^k(\R^n) \ ,
\]
which makes $(T^k(\R^n),+,\otimes_k)$ into an associative algebra with neutral element $1_k \defl (1,0,\dots,0)$, see \cite[Proposition~7.4]{FV}. Similarly one defines the infinite tensor algebra $T^\infty(\R^n) \defl \prod_{i=0}^\infty(\R^n)^{\otimes i}$. 

Given a path $\gamma : [a,b] \to \R^n$ of bounded variation, Chen's theorem states that for $a \leq s < t < u \leq b$
\begin{equation}
\label{eq:chen}
S_k(\gamma)_{s,u} = S_k(\gamma)_{s,t}\otimes_k S_k(\gamma)_{t,u} \ ,
\end{equation}
see e.g.\ \cite[Theorem~7.11]{FV}.

Set $\mathfrak t^k(\R^n) \defl \{g \in T^k(\R^n) : \xi_0(g) = 0\}$. It is shown in \cite[Proposition~7.17]{FV} that the set $1_k + \mathfrak t^k(\R^n) \subset T^k(\R^n)$ together with the truncated tensor product $\otimes_k$ of $T^k(\R^n)$ is a Lie group. Similarly, the set $\mathfrak t^k(\R^n) \subset T^k(\R^n)$ together with the bilinear map
\[
[g,h] \defl g\otimes_k h - h \otimes_k g
\]
is a Lie algebra, \cite[Proposition~7.19]{FV}. By $\mathfrak g^k(\R^n) \subset \mathfrak t^k(\R^n)$ we denote the smallest Lie subalgebra that contains $(\R^n)^{\otimes 1} \subset \mathfrak t^k(\R^n)$. The exponential map $\exp : \mathfrak t^k(\R^n) \to 1_k + \mathfrak t^k(\R^n)$, $\exp(a) \defl 1_k + \sum_{i=1}^k \frac{a^{\otimes i}}{i!}$, maps $\mathfrak g^k(\R^n)$ onto a closed Lie subgroup $G^k(\R^n) \defl \exp(\mathfrak g^k(\R^n))$ of $1_k + \mathfrak t^k$, the \textbf{free nilpotent Lie group of step} $\mathbf k$ \textbf{over} $\pmb{\R^n}$, \cite[Corollary~7.27]{FV}. The straight path $\gamma : [0,1] \to \R^n$ given by $\gamma(t) = tv$ for some $v \in \R^n$ has signature $S_k(\gamma)_{0,1} = \exp(v)$, \cite[Example~7.21]{FV}. Chow's theorem states that for any $g \in G^k(\R^n)$ there exists a piecewise linear path $\gamma : [0,1] \to \R^n$ such that $g = S_k(\gamma)_{0,1}$, see e.g.\ \cite[Theorem~7.28]{FV}. Moreover,
\[
G^k(\R^n) = \{S_k(\gamma)_{0,1}\ :\ \gamma : [0,1] \to \R^n \mbox{ is rectifiable}\} \ ,
\]
\cite[Theorem~7.30]{FV}. The Carnot-Carath\'eodory norm on $G^k(\R^n)$ is defined by
\[
\|g\| \defl \inf \{L(\gamma)\ :\ \gamma : [0,1] \to \R^n \mbox{ is rectifiable and } S_k(\gamma)_{0,1}=g\} \ .
\]
This infimum is achieved by some rectifiable path $\gamma$, see \cite[Theorem~7.32]{FV}, and the \textbf{Carnot-Carath\'eodory metric} $d_k(g,h) \defl \|g^{-1}\otimes_k h\|$ is a left invariant geodesic metric on $G^k(\R^n)$, see e.g.\ \cite[Definition~7.41]{FV}. By construction it is clear that $\pi_{k}^l : G^l(\R^n) \to G^k(\R^n)$ is $1$-Lipschitz for any $1 \leq k \leq l$.

Let $\gamma : [a,b] \to \R^n \cong G^1(\R^n)$ be a rectifiable path and let $g \in G^k(\R^n)$ with $\pi_{1}^k(g) = \gamma(a)$. We define $\gamma_k : [a,b] \to G^k(\R^n)$ by
\[
\gamma_k(t) \defl g \otimes_k S_k(\gamma)_{a,t} \ ,
\]
for $t \in [a,b]$. Clearly $\pi_{1}^k \circ \gamma_k = \gamma$, $\gamma_k(0) = g$ and by \cite[Proposition~7.59]{FV} and the left invariance of $d_k$ we obtain that
\[
L(\gamma_k) = L(t \mapsto S_k(\gamma)_{a,t}) = L(\gamma) \ .
\]
This implies that all the projections
\[
\pi_{k+1} : (G^{k+1}(\R^n),d_{k+1},1_{k+1}) \to (G^k(\R^n),d_k,1_k)
\]
have the unique path lifting property as defined earlier in Definition \ref{def:uniquelift}. The fact that these projections also have the continuous path lifting property, in the sense of Definition \ref{def:continuouslift}, is certainly known but we could not find a reference. For the sake of convenience we add a proof here.

\begin{lem}
\label{contpath_lem}
For every $k \in \N$, the projection
\[
\pi_{k+1} : (G^{k+1}(\R^n),d_{k+1},1_{k+1}) \to (G^k(\R^n),d_k,1_k)
\]
has the continuous path lifting property. Moreover, each each metric space $G^k(\R^n)$ is geodesic and proper.
\end{lem}

\begin{proof}
The standard topology on $1_k + \mathfrak t^k(\R^n) \subset T^k(\R^n) = \prod_{i=0}^k (\R^n)^{\otimes i}$ is induced by the distance function
\[
\rho(x,y) \defl \max_{i=1,\dots,k}|\xi_i(x)-\xi_i(y)| \ .
\]
In fact, as shown in \cite[Proposition~7.45]{FV}, there is a constant $C \geq 1$ such that
\[
C^{-1}\rho(1_k,g) \leq \max\{d_k(1_k,g),d_k(1_k,g)^k\}\leq C\rho(1_k,g)
\]
for all $g \in G^k(\R^n)$. Thus the topology of $G^k(\R^n)$ agrees with the topology induced by $\rho$ and closed balls in $G^k(\R^n)$ are compact. Thus each $G^k(\R^n)$ is geodesic and proper.

If $f,g : [0,1] \to \R$ are Lipschitz, then we claim that also the function $\iota(t) \defl \int_0^t f\, dg$ is  Lipschitz. Indeed, 
for $0 \leq s \leq t \leq 1$ it holds that
\begin{align}
\nonumber
|\iota(s)-\iota(t)| & = \left|\int_s^t f(t)g'(t)\, dt \right| \leq \int_s^t |f(t)g'(t)| \, dt \\
\nonumber
 & \leq \|f\|_\infty\Lip(g)|t-s| \\
\label{lipschitz_eq}
 & \leq (f(0) + \Lip(f))\Lip(g)|t-s| \ .
\end{align}
Moreover, we claim that if $(f_m)_{m \in \N}$ and $(g_m)_{m \in \N}$ are sequences of Lipschitz functions on $[0,1]$ with $\sup_m \{\Lip(f_m),\Lip(g_m)\} < \infty$ that converge pointwise (and thus uniformly) to $f$ and $g$, respectively, then
\begin{equation}
\label{convergence_eq}
\lim_{m\to\infty}\int_0^1 f_m \, dg_m = \int_0^1 f \, dg \ .
\end{equation}
This follows for example by properties of the Riemann-Stieltjes integral, see \cite{Y} for more general convergence results where the statement above is attributed to Helly \cite{H}. Alternatively, one has another proof of \eqref{convergence_eq} by osserving that  the characteristic function $\chi_{[0,1]}$ induces a $1$-dimensional metric current  on $\R$, see \cite[Example~3.2]{AK}. 

Assume now that $\gamma_m : [0,1] \to \R^n$, $m \in \N$, is a sequence of Lipschitz paths with $\gamma_m(0) = 0$, $\sup_m \Lip(\gamma_m) < \infty$, and $\lim_{m\to\infty} \|\gamma_m - \gamma\|_\infty = 0$. It follows by induction on $k$ using \eqref{lipschitz_eq} and \eqref{convergence_eq} that
\[
\lim_{m\to\infty} \rho(S_k(\gamma_m)_{0,t},S_k(\gamma)_{0,t}) = 0 \ ,
\]
for all $t \in [0,1]$. By the discussion in the first part of the proof this implies that
\[
\lim_{m\to\infty}  d_k(S_k(\gamma_m)_{0,t},S_k(\gamma)_{0,t}) = 0 \ .
\]
The general continuous lifting property now follows from the left invariance of the metric and the unique path lifting property for the projections $\pi_{1}^k : G^k(\R^n) \to G^1(\R^n) \cong \R^n$.
\end{proof}

As a consequence of the results stated so far and the uniqueness of signatures, a result by Hambly and Lyons \cite{HL}, the inverse limit $G^\infty(\R^n)$ is a metric tree. Here are the details. For a precise definition of topological and metric trees we refer to the beginning of the next section. Let $S$ be the set of signatures of rectifiable paths, i.e., $S$ is the collection of all elements
\[
\left(1,\int_{0 < t < 1} d\gamma_t, \dots, \int_{0 < t_1 < \cdots < t_k < 1} d\gamma_{t_1} \otimes \cdots \otimes d\gamma_{t_k}, \dots \right) \in \prod_{i=0}^\infty(\R^n)^{\otimes i} \ ,
\]
where $\gamma : [0,1] \to \R^n$ ranges over all rectifiable paths in $\R^n$. Also here Chen's equality \eqref{eq:chen} holds, and thus $S$ naturally has a group structure $(S,\otimes,1)$, where $1$ is the element $(1,0,0,\dots)$ in $\prod_{i=0}^\infty(\R^n)^{\otimes i}$. 

\begin{thm}
\label{tree_thm}
There is a natural group isomorphism identifying the space of signatures $S$ with $G^\infty(\R^n)$, where $(G^\infty(\R^n),d_\infty,1_\infty)$ is the inverse limit  of the free nilpotent Lie groups $(G^k(\R^n),d_k,1_k)$ with left invariant Carnot-Carath\'eodory metrics $d_k$ as obtained in Section~\ref{section:inverselimit}. Further, the metric space $G^\infty(\R^n)$ is a complete metric tree.
\end{thm}

\begin{proof}
Set $(X_k,\star_k) \defl (G^k(\R^n),1_k)$ for all $k \in \N$ with the metric and projections defined earlier for $G^k(\R^n)$. Now $(x_1,x_2,\dots) \in X_\infty$ if and only if $\pi_{k+1}(x_{k+1}) = x_k$ for all $k \geq 1$ and $\sup_{k\in\N}\|x_k\| < \infty$. Let $\xi : X_\infty \to \prod_{i=0}^\infty(\R^n)^{\otimes i}$ be the map given by
\[
\xi(x_1,x_2,\dots) \defl (1,\xi_1(x_1),\xi_2(x_2),\dots) \ .
\]
We claim that $\xi$ is a group isomorphism onto $S$. First if $x \in \prod_{i=0}^\infty(\R^n)^{\otimes i}$ is the signature $S_\infty(\gamma)_{0,1}$ of a rectifiable path $\gamma : [0,1] \to \R^n$, then
\[
\|S_k(\gamma)_{0,1}\| \leq L(t \mapsto S_k(\gamma)_{0,t}) = L(\gamma) < \infty \ ,
\]
and hence $x$ is in the image of $\xi$. On the other side assume that $x = \xi(x_\infty)$ for some $x_\infty \in X_\infty$. Because each $\pi_k$ has the continuous path lifting property and each $X_k$ is proper and geodesic by Lemma~\ref{contpath_lem}, it follows from Proposition~\ref{geodesic_prop} that $X_\infty$ is a geodesic metric space. So let $\gamma : [0,1] \to X_\infty$ be a path of finite length with $\gamma(0) = 1_\infty$ and $\gamma(1) = x_\infty$. Because each $\pi^k_1 : X_k \to X_1 = \R^n$ has the unique path lifting property with lifts given by the step $i$ signature of a path, it holds that $x_k = S_k(\pi_1^\infty\circ\gamma)_{0,1}$ for all $k \in \N$. Thus $x = S_\infty(\pi_1^\infty \circ \gamma)_{0,1}$. It follows that $\xi(X_\infty) = S$. By construction it is clear that $\xi$ is injective. This implies that $\xi : X_\infty \to S$ is a bijection. By the construction of $\xi$, the group structure given on $X_\infty$ as defined in Lemma~\ref{grouplem} and the fact that the truncation operator $S \to G^k(\R^n)$ is a group homomorphism it is not hard to check that $\xi : X_\infty \to S$ is a group homomorphism and because $\xi$ is also injective, it is a group isomorphism.

We already know from Proposition~\ref{geodesic_prop} that $X_\infty$ is a complete geodesic metric space. Next we show that $X_\infty$ is a metric tree. Assume that $\gamma_1,\gamma_2 : [0,1] \to X_\infty$ are two geodesics with the same start and end points that are parametrized proportional to arc length, i.e., $L(\gamma_k|_{[0,t]}) = ct$ for some $c > 0$. We want to show that $\gamma_1 = \gamma_2$. By Lemma~\ref{grouplem}, the space $X_\infty$ is a group and $d_\infty$ is left invariant. We therefore can assume that $\gamma_1(0) = \gamma_2(0) = 1_\infty$. As obtained above, $\xi(\gamma_k(t)) = S_\infty(\pi_1^\infty\circ\gamma_k)_{0,t}$ for $k=1,2$ and $t \in [0,1]$. Because $\gamma_k(t)$ is injective, the paths $\pi_1^\infty\circ\gamma_k$ in $\R^n$ are tree reduced, see \cite{HL} for the definition. Because $S_\infty(\pi_1^\infty\circ\gamma_1)_{0,1} = S_\infty(\pi_1^\infty\circ\gamma_2)_{0,1}$, the uniqueness of signatures result due to Hambly and Lyons \cite{HL} implies that $\pi_1^\infty \circ \gamma_1$ and $\pi_1^\infty \circ \gamma_2$ are reparametrizations of each other. Because $\pi_1^\infty : X_\infty \to X_1$ has the unique path lifting property by Lemma~\ref{uniqueliftinfinity} it follows from Lemma~\ref{uniqueliftproperties} that $L(\pi_1^\infty \circ \gamma_1|_{[0,t]}) = L(\pi_1^\infty \circ \gamma_2|_{[0,t]}) = ct$ for all $t \in [0,1]$. Thus $\pi_1^\infty \circ \gamma_1 = \pi_1^\infty \circ \gamma_2$ and therefore also $\gamma_1 = \gamma_2$ again by Lemma~\ref{uniqueliftinfinity}. It follows that $X_\infty$ is a metric tree.
\end{proof}

Because of Proposition~\ref{geodesic_prop} it follows that any tree reduced path in $\R^n$ is the projection of a geodesic in $(G^\infty(\R^n),d_\infty)$. Together with the second part of Proposition~\ref{geodesic_prop}, the next result shows that paths in $\R^n$ can actually be approximated by projections of geodesics in some $G^k(\R^n)$.

\begin{cor}
\label{approx_cor}
For any (continuous) path $\gamma : [0,1] \to \R^n \cong G^1(\R^n)$ and any $\epsilon > 0$ there is some $k \geq 1$ and a geodesic $\gamma_k : [0,1] \to G^k(\R^n)$ such that $L(\gamma_k) \leq L(\gamma)$, $\pi^k_1(\gamma_i(\delta)) = \gamma(\delta)$ for $\delta = 0,1$ and 
\[
\|\gamma - \pi^k_{1} \circ \gamma_k\|_\infty < \epsilon \ .
\]
\end{cor}

\begin{proof}
Note first that $\gamma$ can be uniformly approximated by tree reduced rectifiable paths with length bounded from above by $L(\gamma)$ while fixing the endpoints. So without loss of generality we may assume that $\gamma$ already is a tree reduced rectifiable path. Let $(G^\infty(\R^n),1_\infty)$ be, as before,  the inverse limit of the sequence of pointed metric spaces $(G^k(\R^n),1_k)_{k\in\N}$. Since $\pi^\infty_1 : G^\infty(\R^n) \to \R^n$ is surjective by Lemma~\ref{firstlem}, there is a point $x_\infty \in G^\infty(\R^n)$ that projects to $\gamma(0) \in \R^n$. Due to Lemma~\ref{uniqueliftinfinity}, $\pi^\infty_1$ has the unique path lifting property and hence $\gamma$ lifts to a path $\gamma_\infty : [0,1] \to G^\infty(\R^n)$ starting at $x_\infty$. Since $\gamma$ is tree reduced, the curve $\gamma_\infty$ is injective and hence by Theorem~\ref{tree_thm} it is a geodesic (not necessarily parametrized proportional to arc length). Since $\gamma_\infty$ is the unique geodesic (up to reparametrizations) connecting its end points it now follows from Proposition~\ref{geodesic_prop} that there is an increasing sequence $\sigma : \N \to \N$ and geodesics $\gamma_{\sigma(k)}$ in $G^{\sigma(k)}(\R^n)$ for each $k \in \N$ such that $\pi^{\sigma(k)}_1 \circ \gamma_{\sigma(k)} : [0,1] \to G^{\sigma(k)}(\R^n) \to \R^n$ converges uniformly to $\pi^\infty_1 \circ \gamma_\infty = \gamma$.
\end{proof}

Instead of paths in $\R^n$ in the corollary above the analogous statement holds for paths in $G^k(\R^n)$ which can be approximated by projections of geodesics in $G^i(\R^n)$ for some $i > k$. The proof is essentially the same.

\begin{rem}
	\label{rem:limitcontpathlifing}
Although the projection $\pi^\infty_1 : X_\infty \to X_1$ has the unique path lifting property if each $\pi_i$ has it by Lemma~\ref{uniqueliftinfinity}, the map $\pi^\infty_1$ in general does not inherit the continuous path lifting property. This can be seen from the sequence of rectifiable paths $\gamma_n : [0,2] \to \R^2$ given by
\[
\gamma_n(t) \defl
\left\{
	\begin{array}{ll}
		(1-t,0)	  & \mbox{if } t \in [0,1] \\
		(t-1,\tfrac{1}{n}(t-1)) & \mbox{if } t \in [1,2]
	\end{array}
\right.
\]
and their lifts $\bar \gamma_n$ to $G^\infty(\R^2)$. Each $\bar \gamma_n$ is injective and because $(G^\infty(\R^2),d_\infty)$ is a metric tree by Theorem~\ref{tree_thm}, it holds that
\[
d_\infty(\bar\gamma_n(0),\bar\gamma_n(2)) = L(\bar \gamma_n) \geq L(\gamma_n) \geq 2 \ .
\]
But the limit path $\gamma : [0,2] \to \R^2$ has trivial signature, i.e., $\bar \gamma(0) = \bar \gamma(2)$ for its lift into the limit space $(G^\infty(\R^2),d_\infty)$. Thus $\pi^\infty_1 : G^\infty(\R^2) \to \R^2$ does not have the continuous path lifting property.
\end{rem}

\section{Trees as topological groups}

Let us recall the following definitions. An {\em arc} in a topological space $X$ is a set $I \subset X$ homeomorphic to $[0,1]$. By abuse of notation we denote by $\partial I$ the end points of $I$. The space $X$ is said to be {\em arcwise connected} if for any two different points $x,y \in X$ there is an arc connecting $x$ with $y$, i.e.,  there is an arc $I \subset X$ such that $\partial I = \{x,y\}$.

A \textbf{topological tree} is a metrizable topological space $X$ with the following two properties:
\begin{enumerate}
	\item $X$ is uniquely arcwise connected, i.e., for any two different points $x,y \in X$ there is a unique arc in $X$ connecting $x$ with $y$. We will denote this arc by $[x,y]$.
	\item $X$ is locally arcwise connected, i.e., for any $x \in X$ and any neighbourhood $U$ of $x$ there is an arcwise connected open neighbourhood $V$ of $x$ contained in $U$. 
\end{enumerate}

A \textbf{metric tree} is a uniquely arcwise connected geodesic metric space.

It is clear that any metric tree is a topological tree. The main result of \cite{MO} states that the converse holds too.

\begin{thm}[Theorem~5.1 in \cite{MO}]
\label{toptreeequiv_thm}
$X$ is a topological tree if and only if $X$ is homemorphic to a metric tree.
\end{thm}

The main goal of this section is to prove that nontrivial topological trees cannot have the structure of a topological group. This is implied by the following theorem. As a consequence we obtain that right translations are not continuous in the limit space $G^\infty(\R^n)$, or more generally, the space $G^\infty(\R^n)$ is not homeomorphic to a topological group.

\begin{thm}
\label{treenogroup_thm}
Assume that $X$ is a topological tree with a group structure $(\cdot,{}^{-1},e)$ such that all left translations $x \mapsto g\cdot x$ and right translations $x \mapsto x \cdot g$ are continuous. If $X$ contains more than one point, then $X$ is homemorphic to $\R$.
\end{thm}

\begin{proof}
We first show that in case $X$ is not an interval, then there can't be a group structure on $X$ with continuous left and right translations. This is proved by contradiction. Due to Theorem~~\ref{toptreeequiv_thm} we assume that $X$ is a metric tree.

\emph{Claim 1:} If $X$ is not an interval, it contains a tripod. This means that there are different points $x,x_1,x_2,x_3 \in X$ such that $[x,x_i] \cap [x, x_j] = \{x\}$ for $i \neq j$.

\emph{Proof of claim 1:} Assume that $X$ does not contain a tripod. Fix some $x_0 \in X$. Assume first that there is some $\epsilon > 0$ for which the set $A_\epsilon \defl \{x \in X : d(x_0,x) = \epsilon\}$ contains more than two points. Then any three points in $A_\epsilon$ form a tripod, contradicting the assumption. Now assume that $A_\epsilon$ contains only one point for any $\epsilon > 0$. Then $f : X \mapsto \R$ given by $f(x) \defl d(x_0,x)$ is an isometric embedding of $X$ onto some interval in $\R$. Finally assume that there are two different points $x_-,x_+ \in A_\epsilon$ for some $\epsilon > 0$. Let $X_{+}$ be the set of points $x \in X$ for which $[x_0,x_+] \cap [x_0,x] \neq \{x_0\}$. Because $X$ is uniquely arcwise connected this means that $[x_0,x_+] \cap [x_0,x] = [x_0,z]$ for some $z \neq x_0$. If now $x,y \in X_+$, then $[x_0,x]$ and $[x_0,y]$ meet in some initial segment and thus $[x_0,x]\cap[x_0,y] = [x_0,z]$ for some $z \neq x_0$ as before. But since $X$ does not contain a tripod, $x = z$ if $d(x_0,x) \leq d(x_0,y)$ and $y = z$ if $d(x_0,y) \leq d(x_0,x)$. Similarly we define $X_-$ with respect to $x_-$. Note that $X_+ \cap X_-$ are disjoint, otherwise $[x_0,x_+] \cup [x_0,x_-]$ would form a tripod. Thus $f : X \to \R$ defined by
\[
f(x) \defl 
\left\{
	\begin{array}{ll}
			0	  & \mbox{if } x = x_0 \\
		+d(x_0,x) & \mbox{if } x \in X_+ \\
		-d(x_0,x) & \mbox{if } x \in X_-
	\end{array}
\right.
\]
is an isometric embedding of $X$ onto an interval of $\R$. This proves the claim.

So there exists a tripod in $X$ as stated in the claim above. Left (or right) translating these arcs by $x^{-1}$ we may assume that $x = e$ is the center of a tripod. Pick some point $y \in [e,x_1] \setminus \{e,x_1\}$. Note that since left translation by $y$ is a homeomorphism, any set $y \cdot I$ is an arc if and only if $I$ is an arc. Thus, $y \cdot [e,x_i] = [y,y\cdot x_i]$.

\emph{Claim 2:} At most one of the arcs $[y,y\cdot x_1]$, $[y,y\cdot x_2]$ and $[y,y\cdot x_3]$ intersects the segment $[e,y]\setminus \{y\}$ and at most one of these arcs intersects $[y,x_1]\setminus \{y\}$.

This claim is quite clear since if two different $[y,y\cdot x_i]$ and $[y,y\cdot x_j]$ intersect say $[e,y]\setminus \{y\}$, then $[y,y\cdot x_i]\cap[y,y\cdot x_j] = [y,z]$ for some $z \neq y$. But this contradicts that $[y,y\cdot x_i] \cap [y,y\cdot x_j] = \{y\}$.

So there exists some $i \in \{1,2,3\}$ and a sequence $(y_n)_{n \in \N}$ in $[e,x_1] \setminus \{e,x_1\}$ with $\lim_{n\to\infty} y_n = e$ with $[y_n,y_n\cdot x_i] \cap [e,x_1] = \{y_n\}$. It follows that $[x_i,e] \cup [e,y_n] \cup [y_n,y_n\cdot x_i]$ is an arc. Set $U \defl X \setminus \{e\}$ Because $X$ is locally arcwise connected, there exists some arcwise connected open neighbourhood $V$ of $x_i$ contained in $U$. Because right translation by $x_i$ is continuous, $y_n\cdot x_i \in V$ for some large enough $n$. For such an $n$ there exists an arc entirely in $V$ connecting $x_i$ with $y_n\cdot x_i$. But this contradicts that $[x_i,e] \cup [e,y_n] \cup [y_n,y_n\cdot x_i]$ is another arc connecting $x_i$ with $y_n\cdot x_i$.

This shows that $X$ is homeomorphic to an interval. If $X$ has more than one point, then there is a point in $X$ with a neighbourhood homeomorphic to $\R$. Because left translations are homemorphisms of $X$ that act transitively we obtain that all points of $X$ have neighbourhoods homeomorphic to $\R$. As an interval $X$ has thus to be homeomorphic to $\R$.
\end{proof}

The conclusion in the theorem above can be strengthened by saying that $(X,\cdot)$ and $(\R,+)$ are isomorphic as topological groups. This boils down to the fact that the only topological group structure on $\R$ is the standard one.


\end{document}